\documentclass{amsart}
\usepackage{amsmath, amssymb,epic,graphicx,mathrsfs,enumerate}
\usepackage[all]{xy}
\usepackage{color}
\usepackage{comment}

\usepackage{amsthm}
\usepackage{amssymb}
\usepackage{latexsym}
\usepackage{longtable}
\usepackage{epsfig}
\usepackage{amsmath}
\usepackage{hhline}

 \DeclareMathOperator{\frat}{Frat}
\DeclareMathOperator{\ssl}{SL}

  \DeclareMathOperator{\diam}{diam}

\DeclareMathOperator{\sym}{Sym}

\DeclareMathOperator{\dist}{dist}

\DeclareMathOperator{\z}{{Z}}
\newtheorem{thm}{Theorem}
\newtheorem{cor}[thm]{Corollary}
 \newtheorem{lemma}[thm]{Lemma}
\newtheorem{prop}[thm]{Proposition}

\numberwithin{equation}{section}

\renewcommand{\footnote}{\endnote}
\newcommand{\ignore}[1]{}\makeglossary

\begin{document}
	\bibliographystyle{amsplain}
	\subjclass{ 20D60, 05C25}
	\title[The virtually generating graph of a profinite group]{The virtually generating graph\\ of a profinite group}

	\author{Andrea Lucchini}
	\address{Andrea Lucchini\\ Universit\`a degli Studi di Padova\\  Dipartimento di Matematica \lq\lq Tullio Levi-Civita\rq\rq\\ Via Trieste 63, 35121 Padova, Italy\\email: lucchini@math.unipd.it}
	

\begin{abstract} We consider the graph $\Gamma_{\rm{virt}}(G)$ whose vertices are the elements of a finitely generated profinite group $G$  and where two vertices $x$ and $y$ are adjacent if and only if they topologically generate an open subgroup of $G$. We investigate the connectivity of the graph $\Delta_{\rm{virt}}(G)$ obtained from $\Gamma_{\rm{virt}}(G)$ by removing its isolated vertices.	In particular we prove that for every positive integer $t$, there exists a finitely generated prosoluble group $G$ with the property that $\Delta_{\rm{virt}}(G)$ has precisely $t$ connected components. Moreover we study the graph  $\tilde \Gamma_{\rm{virt}}(G)$,
whose vertices are again the elements of $G$ and where two  vertices are adjacent if and only if there exists a minimal generating set of $G$ containing them. In this case we prove that the subgraph $\tilde \Delta_{\rm{virt}}(G)$ obtained removing the isolated vertices is connected and has diameter at most 3.
\end{abstract}
	\maketitle

\section{Introduction}

The generating graph of a finite group is an interesting notion and has been much studied in the literature. In \cite{pgg} the investigation on this topic has been extended to the context of profinite groups. Let $G$ be a 2-generated profinite group. The generating graph $\Gamma(G)$ is defined
with vertex set $G$ where two elements $x$ and $y$ are connected if and only if $x$ and $y$ (topologically) generate $G.$ An open conjecture is that if $G$ is a finite group, then the subgraph $\Delta(G)$ obtained from $\Gamma(G)$ by
removing its isolated vertices is connected. In \cite{pgg} it is proved that the conjecture is true if $G$ is prosoluble but an example is given showing that it fails for arbitrary finitely
generated profinite groups. Recently, during the workshop \lq Groups with Geometrical and Topological Flavours\rq\ Yiftach Barnea proposed to investigate a related graph. Following his suggestion we define the virtually generating graph $\Gamma_{\rm{inv}}(G)$ associated to a profinite group $G$ as the graph whose vertices are the elements of $G$ and in which two elements $x$ and $y$ are connected if and only if $x$ and $y$ generate (topologically) an open subgroup of $G$.
In particular we address the question whether the graph 
$\Delta_{\rm{virt}}(G)$ obtained from $\Gamma_{\rm{virt}}(G)$ by
removing its isolated vertices is connected.
The answer is negative. Let $G_p=(\ssl(2,2^p))^{\delta_p},$ where $p$ is a prime and $\delta_p$ is the largest positive integer with the property that the direct power  $(\ssl(2,2^p))^{\delta_p}$ can be generated by 2-elements and consider $G=\prod_p G_p.$ In \cite[Section 3]{pgg}, it is proved
that the graph $\Delta(G)$ has $2^{\aleph_0}$ connected components. This is a consequence of the fact that the diameter of
$\Delta(G_p)$ tends to infinity with $p.$ In Section \ref{esempiouno} we will see that also $\Delta_{\rm{virt}}(G)$
 has $2^{\aleph_0}$ connected components. In Section  \ref{esempiodue} we will prove a more surprising results, considering the graph $\Delta_{\rm{virt}}(G)$, when $G$ is prosoluble.  In \cite{diam} it was proved  that if $G$ is a finite soluble group, then $\Delta(G)$ is connected, with diameter at most 3, and this statement has been generalized in \cite[Theorem 3]{pgg} to the case of prosoluble groups. However 
 this result cannot extended to $\Delta_{\rm{virt}}(G).$ Indeed we prove:
 
\begin{thm}\label{tantec}
For every positive integer $t$, there exists a finitely generated prosoluble group $G$ with the property that
$\Delta_{\rm{virt}}(G)$ has precisely $t$  connected components.
\end{thm}

In the second part of the paper we extend to the profinite groups the notion of independence graph, given in \cite{ind} for finite groups, and we introduce the new definition of virtually independence graph of a profinite group.
Let $G$ be a finitely generated profinite group. A generating set $X$ of $G$ is said to be minimal if no proper subset of $X$ (topologically) generates $G$. We denote by $d(G)$ and $m(G)$, respectively, the smallest and the largest cardinality of a minimal generating set of a finite group $G.$ A nice result in universal algebra, the Tarski irredundant basis theorem (see for example \cite[Theorem 4.4]{bs}), implies that if $G$ is a finite group, then, for every positive integer $k$ with $d(G)\leq k\leq m(G),$ $G$ contains an independent generating set of cardinality $k$. In section \ref{tacchino} we will prove that this result can be extended
to the profinite case. In this case we define $m(G)$ as the supremum of the sizes of the minimal generating set of $G$.
\begin{thm}
	Let $G$ be a finitely generated profinite group. For every positive integer $k$ with $d(G)\leq k\leq m(G),$ $G$ contains an independent generating set of cardinality $k$. Moreover $m(G)$ is finite if and only if the Frattini subgroup of $G$ has finite index in $G.$
\end{thm}

 We denote by $\tilde\Gamma(G)$  
the graph whose vertices are the elements $G$ and in which two vertices $x$ and $y$ are joined by an edge if and only if $x\neq y$ and there exists a finite minimal generating set of $G$ containing $x$ and $y$. Roughly speaking,  $x$ and $y$ are adjacent vertices of $\tilde\Gamma(G)$ if they are \lq independent\rq, so we  call $\Gamma(G)$ the {\slshape{independence graph}} of $G$. In Section
\ref{inde}, extending results proved in \cite{ind} in the case of finite groups,  we show that if $G$ is a finitely generated profinite group and $g$ is an isolated vertex of $\tilde\Gamma(G),$ then either $g$ belongs to the Frattini subgroup of $G$ or it is a topological generator of $G$. Moreover the graph obtained from  $\tilde\Gamma(G)$ by removing the isolated vertices is always connect. Also in this case instead of  generating subsets, we may consider virtually generating subsets, i.e. subsets generating an open subgroup of $G$. In particular we 
say that two elements $x$ and $y$ are virtually independent if there exists a minimal generating set of an open subgroup of $G$ containing them. In Section \ref{ultima} we define
the virtually independent graph $\tilde \Gamma_{\rm{virt}}(G)$ of $G$ as the graph whose vertices are the elements of $G$ and in which two vertices are joined by an edge if and only if they are virtually independent. Moreover we denote by  $\tilde \Delta_{\rm{virt}}(G)$ the subgraph induced by the non-isolated vertices of the virtually independent graph of $G$.
We prove in particular the following results:
\begin{thm}Let $G$ be a finitely generated profinite group. If $\tilde \Gamma_{\rm{virt}}(G)$ contains a non-trivial isolated vertex $g,$ then either $g$ is a topological generator of $G$ or one of the following occurs:
	\begin{enumerate}
		\item $G\cong C_{p^n}$ is a cyclic group of $p$-power order and all the vertices of $\tilde \Gamma_{\rm{virt}}(G)$ are isolated.
		\item $G\cong \mathbb Z_p$, the group of the $p$-adic integers, and  all the vertices of $\tilde \Gamma_{\rm{virt}}(G)$ are isolated.
		\item $G\cong Q_{2^n}=\langle x, y \mid x^{2^{n-1}}, y^2=x^{2^{n-2}}, x^{-1}yx=y^{-1}\rangle$ is a generalized quaternion group and $g=y^2$ is the unique non-trivial isolated vertex of  $\tilde \Gamma_{\rm{virt}}(G).$ 
	\end{enumerate}
\end{thm}
\begin{thm}
	If $G$ is a finitely generated profinite group, then $\tilde \Delta_{\rm{virt}}(G)$ is connected  and $\diam(\tilde \Delta_{\rm{virt}}(G))\leq 3.$
\end{thm}
We will exhibit an example, showing that the bound  $\diam(\tilde \Delta_{\rm{virt}}(G))\leq 3 $ given by the previous theorem is best possible.
\section{An example of a profinite group $G$ such that\\ $\Delta_{\rm{inv}}(G)$ has
	$2^{\aleph_0}$ connected components.}\label{esempiouno}

Let $G_p=(\ssl(2,2^p))^{\delta_p},$ where $p$ is a prime and $\delta_p$ is the largest positive integer with the property that the direct power  $(\ssl(2,2^p))^{\delta_p}$ can be generated by 2-elements. The graph $\Delta(G_p)$ is connected for every prime $p,$ and, by \cite[Theorem 1.3]{ele}, there exists an increasing sequence $(p_n)_{n\in \mathbb N}$ of odd primes, such that $\diam(\Delta(G_{p_n}))\geq 2^n$ for every $n\in \mathbb N.$ Consider the cartesian product $$G=\prod_{n\in \mathbb N}G_{p_n},$$ with the product topology, and for any $n\in \mathbb N,$ let $\pi_n: G\to G_{p_n}$ be the projection homomorphism. 

\begin{lemma}\label{cinque}
	Let $x=(x_n)_{n\in \mathbb N},$  $y=(y_n)_{n\in \mathbb N} \in G.$  The subgroup $\langle x, y \rangle$ is open in $G$ if and only if there exists $m\in \mathbb N$ such that $\langle x_n, y_n\rangle=G_{p_n}$ for any $n\geq m.$
\end{lemma}
\begin{proof}
	Let $H=\langle x, y\rangle$ and, for $n\in \mathbb N,$ set $H_n=\pi_n(G).$ Clearly $H\leq \prod_{n\in \mathbb N}H_n$ and $|G:H|\geq |G:\prod_{n\in \mathbb N}H_n|=
	\prod_{n\in \mathbb N}|G_{p_n}:H_n|.$ If $H$ is open in $G$, then $|G:H|$ is finite and therefore $H_n=G_{p_n}$ for all but finitely many $n.$ Conversely, assume that $H_n=G_{p_n}$ for any $n\leq m$ and set $K=\prod_{n < m}G_{p_n}$. If follows from the Goursat's Lemma (see for example \cite[Theorem 5.5.1]{hall}) that $HK=G,$ and so $H$ is open in $G.$
\end{proof}

Let $V_{\rm{virt}}(G)$ be the set of the non-isolated vertices
of $\Gamma_{\rm{virt}}(G).$ Moreover, given $x=(x_n)_{n\in \mathbb N}\in G$, define $\Lambda(x)=\{n\in \mathbb N\mid x_n\in V(G_{p^n})\}.$ 

\begin{cor}
	$x=(x_n)_{n\in \mathbb N}$ is a non-isolated vertex of  $\Gamma_{\rm{virt}}(G)$  if and only if $n\in \Lambda(x)$ for all but finitely many $n\in \mathbb N.$
\end{cor}

\begin{lemma}\label{coco}
	Let $x=(x_n)_{n\in \mathbb N}\in V_{\rm{virt}}(G)$ and let $\Omega_x$ be the connected component of $\Delta_{\rm{virt}}(G)$ containing $x.$ Then $y=(y_n)_{n\in \mathbb N}$ belongs to $\Omega_x$ if and only if
	$$\sup_{n\in \Lambda(x)\cap \Lambda(y)}\dist_{\Delta(G_{p_n})}(x_n,y_n)<\infty.$$
\end{lemma}

\begin{proof}
	Assume that $y=(y_n)_{n\in \mathbb N}\in \Omega_x$ and let $m=\dist_{\Delta(G)}(x,y).$ It follows from Lemma \ref{cinque} that
	$\dist_{\Delta(G_{p_n})}(x_n,y_n)\leq m$  for all but finitely many $n\in  \Lambda(x)\cap \Lambda(y).$
	Conversely assume  $y=(y_n)_{n\in \mathbb N}\in V_{\rm{virt}}(G)$ and $\dist_{\Delta(G_{p_n})}(x_n,y_n)\leq m$  for every $n$ in a cofinite subset $\Lambda$ of  $\Lambda(x)\cap \Lambda(y).$ If $n \in \Lambda$, then there is a path $$x_n=x_{n,0},x_{n,1},\dots,x_{n,\mu_n}=y_n,$$ with $\mu_n\leq m,$ joining $x_n$ and $y_n$ in the graph $\Delta(G_{p_n}).$ For $0\leq i\leq m,$ set
	$$\begin{aligned}
	\tilde x_{n,i}=\begin{cases}x_n &\text{if $n\notin \Lambda$,}\\
	x_{n,i}&\text{if $n\in \Lambda,$ $i < \mu_n$,}\\x_{n,\mu_n}&\text{if $n\in \Lambda,$ $i\geq \mu_n$ and $m-\mu_n$ is even,}\\x_{n,\mu_n-1}x_{n,\mu_n}&\text{if $n\in \Lambda,$ $i\geq \mu_n$ and $m-\mu_n$ is odd.} 
	\end{cases}
	\end{aligned}$$
	Then  $$x=\tilde x_0=(\tilde x_{n,0})_{n\in \mathbb N},\ \tilde x_1=(\tilde x_{n,1})_{n\in \mathbb N},\dots,y=\tilde x_m=(\tilde x_{n,m})_{n\in \mathbb N}
	$$ is a path joining $x$ and $y$ in the graph $\Delta_{\rm{virt}}(G),$ so $y\in \Omega_x.$
\end{proof}

\begin{prop}
	$\Delta_{\rm{virt}}(G)$ has $2^\aleph_0$ different connected components.
\end{prop}
\begin{proof}
	Fix $x=(x_n)_{n\in \mathbb N}\in \Delta_{\rm{virt}}(G).$ Let $\tau$ be
	a  real number with $\tau>1.$ Since $$\diam(\Delta(G_{p_n}))\geq 2^n\geq 1+\lfloor n/\tau \rfloor,$$ for every $n\in \mathbb N$
	there exists $y_{\tau,n}\in G_{p_n}$ such that
	$\dist_{\Delta(G_{p_n})}(x_n,y_{\tau,  n})=1+\lfloor n/\tau \rfloor.$ If $\tau_2>\tau_1,$ then
	$$\dist(y_{\tau_2,n},y_{\tau_1,n})\geq
	\dist(x_n,y_{\tau_1,n})-\dist(x_n,y_{\tau_2,n})=\lfloor n/\tau_1 \rfloor - \lfloor n/\tau_2 \rfloor$$
	tends to infinity with $n,$ so, by Lemma \ref{coco}, $\Omega_{y_{\tau_1}}\neq \Omega_{y_{\tau_2}}.$
\end{proof}

\section{Proof of Theorem \ref{tantec}}\label{esempiodue}
We consider the set $\Omega$ of the elements $(x_1,x_2,\dots,x_{2t-1},x_{2t})\in \mathbb F_2^{2t}$ such that $(x_{2i-1},x_{2i})\neq (0,0)$ for $1\leq i \leq t.$ To any $\omega\in \Omega$ we associate a different odd prime number $p_\omega.$ Let
$H=\langle y_1, y_2,\dots,y_{2t-1},y_{2t}\rangle$ be an elementary abelian 2-group of rank ${2t}.$ For any $\omega=(x_1,x_2,\dots,x_{2t-1},x_{2t})\in \Omega$, we define an action of $H$ on the direct product $N_\omega={(\mathbb Z_{p_\omega})}^3$ of three copies of the group $\mathbb Z_{p_\omega}$ of the $p_\omega$-adic integers as follows:
$$(z_1,z_2,z_3)^{y_i}=(z_1,z_2,(-1)^{x_i}\cdot z_3).$$
We consider the semidirect product 
$$G=\left(\prod_{\omega\in \Omega}N_\omega\right)\rtimes H.$$
Notice that $G$ is a prosoluble group that can be generated by $2t$ elements, as it follows easily from the following lemma.
\begin{lemma}\label{ro1ro2}
Let $\omega=(x_1,x_2,\dots,x_{2t-1},x_{2t}) \in \Omega$ and  $v_{\omega}=(1,0,1),$ $w_\omega=(1,0,-1)\in N_\omega.$ Then
	$N_\omega \leq \cap_{1\leq i \leq t}\langle v_\omega y_{2i-1}, w_\omega y_{2i}\rangle.$
\end{lemma}
\begin{proof}
	We prove $N_\omega \leq \langle v_\omega y_1, w_\omega y_2 \rangle$ (the other inclusions $N_\omega \leq \langle v_\omega y_{2i-1}, w_\omega y_{2i}\rangle$ can  be proved with a similar argument). Let $\rho_1=v_\omega y_1$ and $\rho_2=w_\omega y_2$. Notice that $$\rho_1^2=(2,0,0^{x_1}\cdot 2),\ \rho_2^2=(2,0,-0^{x_2}\cdot 2),\ [\rho_1,\rho_2]=\begin{cases}(0,0,4) &\text { if }(x_1,x_2)=(1,1)\\(0,0,-2) &\text{otherwise}. \end{cases}$$
	Since $p_\omega\neq 2,$ it follows $N_\omega=\langle (1,0,0), (0,1,0), (0,0,1)\rangle \leq \langle \rho_1, \rho_2\rangle.$
\end{proof}
Now, for $1\leq i \leq t,$ set 
$$\sigma_{2i-1}=\left((v_\omega)_{\omega\in \Omega}\right)y_{2i-1},
\quad \sigma_{2i}=\left((w_\omega)_{\omega\in \Omega}\right)y_{2i}.$$
It follows immediately from Lemma \ref{ro1ro2} that $\langle \sigma_1, \sigma_2\rangle, \dots,\langle\sigma_{2t-1}, \sigma_{2t-2}\rangle$ are open subgroups of $G$ (with index $2^{t-2}$) and therefore $(\sigma_1,\sigma_2),\dots, (\sigma_{2t-1}, \sigma_{2t-2})$ are edges of $\Gamma_{\rm{virt}}(G).$
Now let $$N=\prod_{\omega \in \Omega}N_\omega\ \text { and } \ \Sigma_i=\{nh \in G \mid n\in N \text { and } h \in \langle y_{2i-1}, y_{2i}\rangle\} \text { for $1\leq i \leq t$}.$$
\begin{prop}\label{almeno2}
	Suppose that $g_1$ and $g_2$ are two  adjacent vertices of $\Gamma_{\rm{virt}}(G).$ Then  $\{g_1,g_2\} \subseteq \Sigma_i$ for some $1\leq i \leq t.$ 
\end{prop}
\begin{proof}
	Suppose that $g_1=n_1h_1,$ $g_2=n_2h_2$ with $n_1, n_2 \in N,$ $h_1=\prod_{1\leq j\leq 2t}y_j^{a_j}$ and $h_2=\prod_{1\leq j\leq 2t}y_j^{b_j}$.
	We claim that the system 
	\begin{equation}\label{sistema}
	\begin{cases}
	a_1x_1+\dots+a_{2t}x_{2t}\!&=\ 0\\
	b_1x_1+\dots+b_{2t}x_{2t}\!&=\ 0
	\end{cases}
	\end{equation}
	where the coefficients are viewed as elements of $\mathbb F_2$, has no solution in $\Omega.$
	Indeed if $\omega=(x_1,\dots,x_{2t})$ is a solution of the previous system, then $h_1$ and $h_2$ centralize $N_{\omega}$ and
	consequently $\langle g_1, g_2 \rangle\cap N_\omega=\langle n_1, n_2\rangle$ is a 2-generated subgroup of $N_\omega={(\mathbb Z_{p_\omega})}^3.$ But this would imply  $$|G:\langle g_1,g_2\rangle|\geq |N_\omega \langle g_1,g_2\rangle : \langle g_1,g_2\rangle|=|N_\omega:N_\omega \cap \langle g_1,g_2\rangle|=\infty$$ and therefore $\langle g_1,g_2\rangle$
	is not an open subgroup of $G$. 
	
	For $1\leq i\leq t$, consider the following matrices, with coefficients in $\mathbb F_2:$
	$$C_i=\begin{pmatrix}a_{2i-1}&a_{2i}\\b_{2i-1 }&b_{2i}\end{pmatrix},	\quad D_i:=\begin{pmatrix}a_1&\dots&a_{2i-2}&a_{2i+1}&\dots &a_{2t}\\
	b_1&\dots&b_{2i-2}&b_{2i+1}&\dots &b_{2t}\end{pmatrix}.$$
	Our statement is equivalent to stating that if $C_i\neq 0$, then $D_i=0$.
	First we prove that if $C_i$ is invertible, then $D_i=0.$
Let $\Omega^*$ be the set of the elements $(r_1,r_2,\dots,r_{t-3},r_{t-2})\in \mathbb F_2^{2(t-1)}$ such that $(r_{2i-1},r_{2i})\neq (0,0)$ for $1\leq i \leq t-1.$ 
Assume, by contradiction,  $D_i\neq 0.$ Since $\langle \Omega^*\rangle = \mathbb F_2^{2t-2},$ there exists an element $(x_1,\dots,x_{2i-2},x_{2i+1}\dots,x_{2t})\in \Omega^*$
such that
$$D_i\begin{pmatrix}x_1\\\vdots\\x_{2i-2}\\x_{2i+1}\\\vdots\\ x_{2t}\end{pmatrix}=\begin{pmatrix}z_1\\z_2\end{pmatrix}\neq \begin{pmatrix}0\\0\end{pmatrix}.$$ If $C_i$ is invertible, then there exists $(x_{2i-1},x_{2i})\neq (0,0)$ such that
	$$C\begin{pmatrix}x_{2i-1}\\x_{2i}\end{pmatrix}=\begin{pmatrix}z_1\\z_2\end{pmatrix}$$
	and $(x_1,\dots,x_{2t})\in \Omega$ is a solution of (\ref{sistema}). We remain with the case where none of the matrices $C_1,\dots,C_t$ is invertible. However in this case,
for any $1\leq i \leq t$, there exists $(x_{2i-1},x_{2i})\neq (0,0)$	 such that
	$$C_i\begin{pmatrix}x_{2i-1}\\x_{2i}\end{pmatrix}=\begin{pmatrix}0\\0\end{pmatrix},$$ and again
	$(x_1,\dots,x_{2t})\in \Omega$ is a solution of (\ref{sistema}).
	\end{proof}
For $1\leq i \leq t,$ let $\Gamma_i$ be the set of the non-isolated vertices of $\Gamma_{\rm{virt}}(G)$ contained in $\Sigma_i.$ By the previous lemma, $\Gamma_i$ is a disjoint union of connected components of $\Gamma_{\rm{virt}}(G)$. We are going to prove that indeed $\Gamma_i$ is a connected components. 
\begin{lemma}\label{condi}
Let $g=\left((n_w)_{\omega\in \Omega})\right)h\in \Sigma_i,$ with $h\in H$ and $n_\omega=(z_{1,\omega},z_{2,\omega},z_{3,\omega})\in N_\omega.$ If $g$ is a non-isolated vertex of $\Gamma_{\rm{inv}}(G),$ then following conditions hold:
\begin{enumerate}
	\item $h \neq 1;$
	\item $(z_{1,\omega},z_{2,\omega})\neq (0,0)$ for any $\omega \in \Omega;$
	\item if $h$ centralized $N_\omega,$ then $z_{3,\omega}\neq 0.$
\end{enumerate}
\end{lemma}
\begin{proof}We may assume $i=1.$
 Suppose that there exists $g^*\!=\!\left((n_{\omega}^*)_{\omega \in  \Omega})\right)h^*$ such that $X:=\langle g, g^* \rangle$ is an open subgroup of $G$. By the previous lemma and its proof, $K:=\langle h, h^*\rangle=\langle a_1,a_2\rangle.$ Let $$A_{\omega}:=\{(z_1,z_2,0)\mid z_1,z_2 \in \mathbb Z_p\}\leq N_\omega, \quad B_{\omega}:=\{(0,0,z_3)\mid z_3 \in \mathbb Z_p\}\leq N_\omega.$$ Since $X$ is open in $G$, $(X\cap N_\omega)B_\omega/B_\omega =\langle (z_{1,\omega},z_{2,\omega},0)B_\omega, (z_{1,\omega}^*,z_{2,\omega}^*,0)B_\omega\rangle$ has finite index
	in $N_\omega/B_\omega\cong (\mathbb Z_{p_\omega})^2$ and this implies
	$(z_{1,\omega},z_{2,\omega})\neq (0,0).$ Finally assume 
	$[h,N_\omega]=0$ and set $M_\omega=\left(\prod_{\tilde \omega \neq \omega}N_{\tilde\omega}\right)\times A_\omega.$ Consider $G_\omega:=G/M_\omega.$ We have that $XM_\omega/M_\omega$ is an open subgroup of $G_\omega$ and this implies that
	$\langle (0,0,z_{3,\omega})h, (0,0,z_{3,\omega}^*)h^*\rangle$ has finite index in $B_\omega K \cong\mathbb Z_{p_\omega}\rtimes K.$
	By definition, $[K,B_\omega]\neq 0.$ Since $[h,B_\omega]=0$ and $K=\langle h, h^*\rangle,$ we deduce $[h^*,B_\omega]\neq 0.$ This implies in particular $((0,0,z_{3,\omega})h)^2=1$.
	If $z_{3,\omega}=0,$ then $(0,0,z_{3,\omega})h, (0,0,z_{3,\omega}^*)h^*$
	are two commuting involutions and do not generated a subgroup of finite index of $B_\omega K.$
\end{proof}

\begin{lemma}\label{tante}
	If $g\in \Sigma_i$ satisfies condition (1), (2), (3) of Lemma \ref{isolati}, then $g$ is a non-isolated vertex of $\Gamma_{\rm{inv}}(G)$. In particular if $g_1$ and $g_2$ are two elements of $\Sigma_i$ satisfying (1), (2), (3), then there exists $\tilde g$ in $\Sigma_i$, adjacent to both $g_1$ and $g_2$ in $\Gamma_{\rm{inv}}(G)$.
\end{lemma}
\begin{proof}
We may assume $i=1.$ Let $g_1=\left((n_{1,\omega})_{\omega\in \Omega}\right)h_1$ and $g_2=\left((n_{2,\omega})_{\omega\in \Omega}\right)h_2,$ 
	with $h_1, h_2\in K=\langle a_1,a_2\rangle.$ Choose $\tilde h \in K \setminus \{1,h_1,h_2\}.$ We want to construct $\tilde n_\omega$ in $N_\omega$ such that $
	\left((\tilde n_{\omega})_{\omega\in \Omega}\right)\tilde h$ is  adjacent to both $g_1$ and $g_2$ in $\Gamma_{\rm{inv}}(G)$. It suffices to choose, for any $\omega \in \Omega,$ $\tilde n_\omega$ in $N_\omega$ with the property that $\langle n_{1,\omega}h_1,\tilde n_{\omega}\tilde h\rangle \cap N_\omega$ and $\langle n_{2,\omega}\tilde h_2,\tilde n_{\omega}\tilde h\rangle \cap N_\omega$ have finite index in $N_\omega.$ Suppose $n_{1,\omega}=(z_{11},z_{12},z_{13})$ and $n_{2,\omega}=(z_{21},z_{22},z_{23}).$ We may choose
	$(\tilde z_1, \tilde z_2) \in (\mathbb Z_{p_\omega})^2$ with the property that $\langle (\tilde z_1, \tilde z_2),
(z_{11},z_{12})\rangle$	and $\langle (\tilde z_1, \tilde z_2),
(z_{21},z_{22})\rangle$ have finite index in $(\mathbb Z_{p_\omega})^2.$ For $j\in \{1,2\}$, set $\eta_j=0$ if
$[h_j, N_\omega]=0$ $\eta_j=1$ otherwise. Similarly set $\tilde \eta=0$ if $[\tilde h, N_\omega]=0,$ $\tilde \eta_j=0$ otherwise. Choose $\tilde z_3 \in \mathbb Z_{p_\omega}$ such that $\tilde z_3 \notin \{\tilde \eta(-1)^{\tilde \eta+1}z_{13}, \tilde \eta (-1)^{\tilde \eta+1}z_{23}\}.$ We claim that
$\tilde n_\omega = (\tilde z_1, \tilde z_2, \tilde z_3)$ satisfies the requested property. More precisely, for $j\in \{1,2\}$, the subgroup $\langle (n_{j,\omega}h_j)^2,(\tilde n_{\omega}\tilde h)^2, [n_{1,\omega}h_j,\tilde n_{\omega}\tilde h] \rangle$ has finite index in $N_\omega.$
Indeed we have the following two possibilities:

\noindent 1) $\eta_j=0.$ In this case, by hypothesis, $z_{j3}\neq 0.$ Moreover, since $K=\langle h_j, \tilde h\rangle$, we must have $[N_\omega,\tilde h]\neq 0.$ So
$$\begin{aligned}\langle (n_{j,\omega}h_j)^2,(\tilde n_{\omega}\tilde h)^2, [n_{j,\omega}h_j,\tilde n_{\omega}\tilde h]\rangle=&
\langle (2z_{j1},2z_{j2},2z_{j3}),(2\tilde z_1,2\tilde z_2,0),
(0,0,-2z_{j3} \rangle\\
=&
\langle (z_{j1},z_{j2},0),(\tilde z_1,\tilde z_2,0),
(0,0,z_{j3}) \rangle
\end{aligned}$$
has finite index in $N_\omega.$	

\noindent 2) $\eta_j\neq 0.$ In this case
$$\begin{aligned}&\langle (n_{j,\omega}h_j)^2,(\tilde n_{\omega}\tilde h)^2, [n_{j,\omega}h_1,\tilde n_{\omega}\tilde h]\rangle\\&=
\langle (2z_{j1},2z_{j2},0),(2\tilde z_1,2\tilde z_2,2(1-\tilde \eta)\tilde z_3),
(0,0,2\tilde \eta z_{j3}+(-1)^{\tilde \eta}2\tilde z_3) \rangle
\\&=
\langle (z_{j1},z_{j2},0),(\tilde z_1,\tilde z_2,(1-\tilde \eta)\tilde z_3),
(0,0,\tilde \eta z_{j3}+(-1)^{\tilde \eta}\tilde z_3) \rangle
\end{aligned}$$
has finite index in $N_\omega$ since $\tilde \eta z_{j3}+(-1)^{\tilde \eta}\tilde z_3\neq 0.$
\end{proof}
Combining Proposition \ref{almeno2}, Lemma \ref{condi} and Lemma \ref{tante} we reach the following conclusion.
\begin{prop}\label{isolati}$\Lambda_1,\dots,\Lambda_t$ are the connected components of the graph $\Delta_{\rm{virt}}(G)$.
\end{prop}

\section{Minimal generating sets}\label{tacchino}

Let $G=\langle x_1.\dots,x_d\rangle$ be a finitely generated profinite group. If $Y$ is a minimal generating set of $G,$ then 
$Y$ is finite. Indeed, for any $1\leq i\leq d,$ there exists a finite subset $Y_i$ of $Y$ such that $x_i \in \langle Y_i \rangle.$ But then $G=\langle x_1.\dots,x_d\rangle \leq \langle Y_1,\dots,Y_d\rangle$ so $Y=Y_1\cup \dots \cup Y_d$ is finite.
Let $m(G)=\sup\{|Y| \mid Y \text{ is a minimal generating set of }G\}.$ By \cite[Theorem 5]{min},  $m(G)$ is finite
if and only if $G$ contains only finitely many maximal subgroups,
i.e. if and only if the Frattini subgroup of $G$ has finite index in $G.$


\begin{thm} Let $G$ be a finitely generated profinite group. For every positive integer $k$ with $d(G)\leq k\leq m(G),$ there exists a minimal generating set of $G$ of size $k.$
\end{thm}

\begin{proof}
	Let $k\in \mathbb N$ with $d(G)\leq k \leq m(G).$ By the way in which $m(G)$ is defined, there exists a minimal generating set $X=\{x_1,\dots,x_t\}$ of $G$ with $k\leq t.$ For $1\leq i \leq t,$ let $X_i:=X\setminus \{x_i\}.$ Since ${\langle X_i\rangle}\neq G,$ there exists an open normal subgroup $N_i$ of $G$ such that $\langle X_i\rangle N_i\neq G.$ Set $N:=N_1\cap \dots \cap N_t.$ We have $\langle X_i\rangle N\leq \langle X_i\rangle N_i < G$ for any $1\leq i \leq t,$ so $\{x_1N,\dots,x_tN\}$ is a minimal generating set for $G/N$. By the Tarski irredundant basis theorem \cite[Theorem 4.4]{bs}, since $d(G/N)\leq d(G)\leq k\leq t\leq m(G/N),$ there exists a minimal generating set $\{y_1N,\dots,y_kN\}$ of $G/N.$ By the profinite version of the Gasch\"utz Lemma (see for example \cite[15.30]{fj}), there exists $n_1,\dots,n_k\in N$ such that $\{y_1n_1,\dots,y_un_k\}$ is a (minimal) generating set of $G.$
\end{proof}

\section{The independence graph}\label{inde}
 We will denote by $\tilde V(G)$ the set of the non-isolated vertices of the independence graph $\tilde \Gamma(G)$ and by $\tilde \Delta(G)$ the subgraph of $\tilde \Gamma(G)$ induced by $\tilde V(G).$ 

\begin{lemma}\label{inizio}Let $G$ be a finitely generated profinite group and
	let $g\in G.$ Then $g$ is isolated in $\tilde \Gamma(G)$ if and only if either $G=\langle g\rangle$ or $g\in \frat(G).$
\end{lemma}
\begin{proof}
	Suppose $g\notin \frat(G).$ There exists an open maximal subgroup $M$ of $G$ with $g\notin M.$ Since $M$ is an open subgroup of a finitely generated profinite group, there exists a finite subset $Y$ of $M$ with $M=\langle Y\rangle.$ The set $X=\{g\}\cup Y$ contains a minimal generating $X$ of $G$ and $g\in X$ (otherwise $G=\langle X\rangle \leq M$). If $X\neq \{g\},$ then $g$ is not isolated in $\tilde \Gamma(G),$ otherwise $\langle g \rangle=G.$
\end{proof}

\begin{lemma}\label{easy}Let $N$ be an open normal subgroup of a finitely generated profinite group $G.$
	If $Y=\{y_1N,\dots,y_tN\}$ is a minimal generating set of $G/N,$   then there exist $n_1,\dots,n_u\in N$ such that
	$\{y_1,\dots,y_t,n_1,\dots,n_u\}$ is a minimal generating set of $G.$
\end{lemma}
\begin{proof}
Let $Z$ be a finite generating set of $N.$	Since $G=\langle Y, Z\rangle,$ $Y\cup Z$ contains a minimal generating set $X$ of $G,$ and the minimality property of $Y$ implies $Y\subseteq X.$
\end{proof}

We will  write $x_1\sim_G x_2$ if $x_1$ and $x_2$ belong to the same connected component of $\tilde\Delta(G).$ The following lemma is an immediate consequence of Lemma \ref{easy}.

\begin{lemma}\label{basic}
	Let $N$ be an open normal subgroup of a finite group $G$ and let $x, y \in G.$ If $xN, yN \in \tilde V(G/N)$ and $xN\sim_{G/N}yN,$ then $x\sim_G y.$
\end{lemma}


\begin{thm}\label{main}
	If $G$ is a finitely generated profinite group, then the graph $\tilde \Delta(G)$ is connected.
\end{thm}
\begin{proof}
	Let $x_1, x_2\in \tilde V(G).$ If follows from Lemma \ref{inizio} that there exists two open normal subgroups, $N_1$ and $N_2$ of $G$ with $x_1N_1\in V(G/N_1)$ and $x_2N_2\in \tilde V(G/N_2).$ Moreover there exists an open normal subgroup $N_3$ with $x_1x_2^{-1}\notin N_3.$ Let $N=N_1\cap N_2\cap N_3.$ Since $x_1N, x_2N$ are distinct elements of $\tilde V(G/N),$ it follows from \cite[Theorem 1]{ind} that $x_1N_1 \sim_{G/N} x_2N_2.$ By Lemma \ref{basic}, $x_1\sim_G x_2.$
\end{proof}

\section{The virtually independence graph}\label{ultima}

Let $G$ be a (finitely generated) profinite group. We say that two elements $x$ and $y$ are virtually independent if there exists a minimal generating set on an open subgroup of $G$ containing them. We define the virtually independent graph $\tilde \Gamma_{\rm{virt}}(G)$ of $G$ as the graph whose vertices are the elements of $G$ and in which two vertices are joined by an edge if and only if they are virtually independent. Moreover we denote by  $\tilde \Delta_{\rm{virt}}(G)$ the subgraph induced by the non-isolated vertices of the virtually independent graph of $G$.

\begin{thm}\label{preced} Let $G$ be a finite group. If $\tilde \Gamma_{\rm{virt}}(G)$ contains a non-trivial isolated vertex $x,$ then either $G=\langle x \rangle$ or one of the following occurs:
\begin{enumerate}
	\item $G\cong C_{p^n}$ is a cyclic group of $p$-power order and all the vertices of $\tilde \Gamma_{\rm{virt}}(G)$ are isolated.
	\item $G\cong Q_{2^n}=\langle a, b \mid a^{2^{n-1}}, b^2=a^{2^{n-2}}, a^{-1}ba=b^{-1}\rangle$ is a generalized quaternion group and $x=b^2$ is the unique non-trivial isolated vertex of  $\tilde\Gamma_{\rm{virt}}(G).$ 
\end{enumerate}
\begin{proof}
	Let $x$ be a non-trivial isolated vertex of  $\tilde \Gamma_{\rm{virt}}(G)$. If $y\in G,$ then $\{x, y\}$ is not a virtually independent subset, so either $x\in \langle y \rangle$ or
	$y\in \langle x \rangle$. In any case, $[x,y]=1.$ So $x\in Z(G).$ Assume $G\neq \langle x \rangle$ and let $p$ be a prime divisor of $|x|.$ Then $\langle x \rangle/\langle x^p \rangle$ is the unique minimal subgroup of $G/\langle x^p \rangle$ (if $\langle y \rangle/\langle x^p \rangle$ were another minimal subgroup of $G,$ then $x$ and $y$ would be virtually independent). In particular $G/\langle x^p \rangle$ is a non-trivial $p$-group and $p$ is the unique prime divisor of $|x|.$ All the elements of order $p$ in $G$ belong to $\langle x\rangle,$ so $G$ contains a unique minimal subgroup
	and therefore $G$ is either cyclic of $p$-power order or generalized quaternion (see \cite[5.3.6]{rob}).
\end{proof}
\end{thm}

\begin{lemma}\label{ken}
	Let $G$ be a finitely generated profinite group and let $N$ be an open normal subgroup of $G$. If $xN$ and $yN$ are adjacent vertices of $\tilde \Gamma_{\rm{virt}}(G/N),$ then $x$ and $y$ are adjacent vertices of $\tilde \Gamma_{\rm{virt}}(G).$
\end{lemma}
\begin{proof}
Let $\{n_1,\dots,n_t\}$ be a finite generating set of $N$ and set
	$A:=\{x,y,n_1,\dots,n_t\}.$ Notice that $H:=\langle A\rangle $ is an open subgroup of $G$ and that $A$ contains a minimal generating set $B$ of $H$. Since $xN$ and $yN$ are adjacent vertices of $\tilde \Gamma_{\rm{virt}}(G/N),$ 
	$\langle x, y\rangle N \neq \langle x \rangle N\neq \langle  y\rangle N$ and therefore
	 $x, y \in B$ and $x$ and $y$ are adjacent in $\Gamma_{\rm{virt}}(G).$ 
\end{proof}

\begin{thm}
Let $G$ be a finitely generated profinite group. If $G$ is infinite and $\tilde \Gamma_{\rm{virt}}(G)$ contains a non-trivial isolated vertex, then $G$ is procyclic.
\end{thm}
\begin{proof}
	Assume that $x$ is a non-trivial isolated vertex in $\tilde \Gamma_{\rm{virt}}(G)$ and that $G$ is not topologically generated by $x.$ Let $\mathcal N$ be the set of the open normal subgroups $N$ of $G$ with the property that $x\notin N$ and $G\neq \langle x\rangle N.$ Let $N\in \mathcal N.$ By Lemma \ref{ken}, $xN$ is an isolated vertex of $\tilde \Gamma_{\rm{virt}}(G/N),$ so, by Theorem \ref{preced}, $G/N$ is either a cyclic $p$-group or a generalized quaternion group. Assume that $G/M$ is a generalized quaternion group for some $M\in \mathcal N.$ Since $\cap_{N\in \mathcal N}N=1$
	and no proper epimorphic image of a generalized quaternion group is generalized quaternion, we would have $M=1$ and 
	 $G$ would be finite, against our assumption. So $G/N$ is cyclic  for every $N\in \mathcal N,$ and therefore $G$ is procyclic. 
\end{proof}

\begin{lemma}\label{casofinito}
	If $G$ is a finite  group, then the graph $\tilde \Delta_{\rm{virt}}(G)$ is connected and $\diam(\tilde \Delta_{\rm{virt}}(G))\leq 3.$
\end{lemma}
\begin{proof}
	We prove the statement by induction. For this purpose, notice that it follows from Lemma \ref{ken} that if $N$ is a normal subgroup of $G$ and  $x_1N$ and $x_2N$ are adjacent vertices of
	$\tilde \Delta_{\rm{virt}}(G/N),$ then $x_1$ and $x_2$ are adjacent in
	$\tilde \Delta_{\rm{virt}}(G)$ and $\dist_{\tilde \Delta_{\rm{virt}}(G)}(x_1,x_2)\leq \dist_{\tilde \Delta_{\rm{virt}}(G/N)}(x_1N,x_2N).$
	Assume that $x_1$ and $x_2$ are distinct vertices of 	$\tilde \Delta_{\rm{virt}}(G)$. If neither $\langle x \rangle \leq
	\langle y\rangle$ nor $\langle y \rangle \leq \langle x \rangle,$ then $\{x,y\}$ is a minimal generating set of 
	$\langle x, y \rangle$ and $\dist_{\Delta_{\rm{virt}}(G)}(x,y)=1.$ 
		So we may assume 
	$\langle x \rangle \leq
	\langle y\rangle$.
	Let $|x|=p_1^{a_1}\cdots p_t^{a_t}, |y|=p_1^{b_1}\cdots p_t^{b_t}, |G|=p_1^{c_1}\cdots p_t^{c_t},$ 
	with $a_i\leq b_i\leq c_i$ for $1\leq i \leq t.$ First assume $t\neq 1.$ If $b_j=0$ for some $1\leq j\leq t,$  let $g$ be an element of $G$ of order $p_j$: then $g$ is adjacent to $x$ and $y$ and $\dist_{\tilde \Delta_{\rm{virt}}(G)}(x,y)\leq 2.$ 
	So we may assume $b_j\neq 0$ for every $1\leq j\leq t.$ Since $y$ is not isolated in $\tilde \Delta_{\rm{virt}}(G)$, we have $G\neq \langle y\rangle$. In particular there exists a prime $p$ and a $p$-element $g\in G$ with $g\notin \langle y\rangle.$ It is not restrictive to assume $p=p_t.$ If $a_i\neq 0$ for some $1\leq i \leq t-1,$ then $g$ is adjacent to $x$ and $y$ and $\dist_{\tilde\Delta_{\rm{virt}}(G)}(x,y)\leq 2.$ Otherwise
	$x - y^{p_t^{b_t}} - g - y$ is a path in $\tilde \Delta_{\rm{virt}}(G)$ and $\dist_{\tilde\Delta_{\rm{virt}}(G)}(x,y)\leq 3.$ We remain with the case when $G$ is a $p$-group. If there exists a minimal subgroup $\langle g\rangle$ of $G$ not contained in $\langle x \rangle,$ then $x$ and $y$ are adjacent to $g$. So we may assume that $G$ has a unique minimal subgroup, and consequently $G$ is either a cyclic $p$-group or a generalized quaternion group. We may exclude the first possibility, since in that case all the vertices of 
	$\tilde \Gamma_{\rm{virt}}(G)$ are isolated. In the second case, $G$ has a unique minimal subgroup, say $N,$ and $xN, yN$ are distinct non-isolated vertices of $G/N$ (which is a dihedral group), so $\dist_{\Delta_{\rm{virt}}(G)}(x,y)\leq \dist_{\tilde \Delta_{\rm{virt}}(G/N)}(x,y)\leq 3$ by induction.
\end{proof}

The bound $\diam(\tilde\Delta_{\rm{virt}}(G))\leq 3$ is best possible. 
Consider $$G=\langle a, b \mid a^4=1, b^3=1, b^a=b^{-1}\rangle$$ 
and let $x=a^2,$ $y=a^2b.$ The elements of $G$ that are adjacent to $y$ in $\tilde \Gamma_{\rm{virt}}(G)$ are precisely those of the form
$a^ib^j$ with $i$ and odd integer. If $g$ is one of these elements, then $g^2=a^2=x,$ so $g$ is not adjacent to $x$ and therefore $\diam(\tilde\Delta_{\rm{virt}}(G))\geq 3$ 

\begin{thm}
	If $G$ is a finitely generated profinite group, then $\tilde\Delta_{\rm{virt}}(G)$ is connected  and $\diam(\tilde\Delta_{\rm{virt}}(G))\leq 3.$
\end{thm}
\begin{proof}
By Proposition \ref{casofinito}, we may assume that $G$ is infinite and that $G$ is not a procyclic pro-$p$-groups (otherwise all the vertices of  $\tilde\Gamma_{\rm{virt}}(G)$ are isolated).
	Assume that $x$ and $y$ are non-isolated distinct vertices of 	$\tilde\Delta_{\rm{virt}}(G)$. There exists three open normal subgroups $N_1,$ $N_2$ and $N_3$ with $x\notin N_1,$ $y\notin N_2$ and $xy^{-1}\notin N_3.$ Set $M=N_1\cap N_2\cap N_3.$ Then $xM$ and $yM$ are two different vertices non-trivial vertices of $G/M.$ We may choose $M$ in such a way that $G/M$ is neither a cyclic $p$-group nor a generalized quaternion group (this is because a proper epimorphic image of a generalized quaternion group is not generalized quaternion). But then $xM$ and $yM$ are non-isolated vertices of $\tilde \Gamma_{\rm{virt}}(G/M)$, so by 
Lemmata \ref{ken} and \ref{casofinito} 	 $\dist_{\tilde\Delta_{\rm{virt}}(G)}(x,y)\leq \dist_{	\tilde\Delta_{\rm{virt}}(G/N)}(xN,yN)\leq 3.$
 \end{proof}


\begin{thebibliography}{99}

	
\bibitem{bs} S. Burris and H. P.  Sankappanavar,
A course in universal algebra. 
Graduate Texts in Mathematics, 78. Springer-Verlag, New York-Berlin, 1981. xvi+276 pp. ISBN: 0-387-90578-2.
	
	\bibitem{ele}	E. Crestani and A. Lucchini, The non-isolated vertices in the generating graph of a direct power of simple groups, J. Algebraic Combin. 37 (2013), no. 2, 249--263.
	
\bibitem{fj} M. Fried and M. Jarden, Field Arithmetic, Springer-Verlag, Berlin, 1986.

\bibitem{hall} M. Hall,  The theory of groups, The Macmillan Co., New York, N.Y. 1959 xiii+434 pp. 
\bibitem{min} A. Lucchini, The largest size of a minimal generating set of a finite group, Arch. Math. (Basel) 101 (2013), no. 1, 1--8.

 \bibitem{diam} A. Lucchini, {The diameter of the generating graph of a finite soluble group}, J. Algebra {492} (2017), 28--43.

\bibitem{pgg} A. Lucchini, The generating graph of a profinite  group,  Arch. Math. (2020) https: //doi.org/10.1007/s00013-020-01502-y.

\bibitem{ind} A. Lucchini, The independence graph of a finite group,   Monatsh. Math. (2020) https: //doi.org/10.1007/s00605-020-01445-0.


	




	
	\bibitem{rob} D. Robinson, A course in the theory of groups. Second edition. Graduate Texts in Mathematics, 80, Springer-Verlag, New York, 1996.
	
	
	
	
	
\end{thebibliography}
\end{document}